\documentclass[12pt,a4paper]{article}
\usepackage{amsfonts,amssymb}
\usepackage{theorem}
{\theorembodyfont{\rm}
  \newtheorem{defi}{Definition}[section]
  \newtheorem{rem}[defi]{Remark}
  \newtheorem{exa}[defi]{Example}
  \newtheorem{exas}[defi]{Examples}}
  \newtheorem{lem}[defi]{Lemma}
  \newtheorem{prop}[defi]{Proposition}
  \newtheorem{thm}[defi]{Theorem}
  
\newcommand{\CC}{{\mathbb C}}

\newcommand{\PP}{{\mathbb P}}
\newcommand{\QQ}{{\mathbb Q}}
\newcommand{\RR}{{\mathbb R}}

\newcommand{\ZZ}{{\mathbb Z}}

\newcommand{\cG}{{\mathcal G}}
\newcommand{\cR}{{\mathcal R}}

\newcommand{\ann}{{\mathrm{ann}}}
\newcommand{\Aut}{{\mathrm{Aut}}}

\newcommand{\End}{{\mathrm{End}}}

\newcommand{\rad}{{\mathrm{rad}}}

\newcommand{\id}{{\mathrm{id}}}
\newcommand{\dis}%{{\mathrel{\vartriangle}}}
                 {{\mathrel{\scriptstyle{\triangle}}}}
\newcommand{\notdis}{{\not\!\!\dis}}
\newcommand{\eps}{{\varepsilon}}

\newcommand{\GL}{{\mathrm{GL}}}
\newcommand{\GE}{{\mathrm{GE}}}
\newcommand{\E}{{\mathrm{E}}}

\newcommand{\M}{{\mathrm{M}}}
\let\phi=\varphi

\let\theta=\vartheta
\newcommand{\DelimArray}[4]{\left#1\begin{array}{*{#3}{r}}#4\end{array}\right#2}
\newcommand{\SDelimArray}[4]{\hbox{\scriptsize\arraycolsep=.5\arraycolsep
  $\left#1\!\!\begin{array}{*{#3}{r}}#4\end{array}\!\!\right#2$}}

\newcommand{\Mat}{\DelimArray()}
\newcommand{\SMat}{\SDelimArray()}

\newenvironment{proof}%[1]%
    {\begin{trivlist} \item {\sl Proof:}} %#1\/:}}%
    {\/ $\square$ \end{trivlist}}

\date{\normalsize
{\em Dedicated to Armin Herzer  on the occasion of his 70th birthday.}
}
\begin{document}
\title{Projective Representations\\
\vspace{.3cm}
{I. Projective lines over rings}}
\author{Andrea Blunck\thanks{Supported by a Lise Meitner
 Research Fellowship
of the Austrian Science Fund (FWF), project M529-MAT.} \and
{Hans Havlicek}}

\maketitle

%%%%%%%%%%%%%%%%%%%%%%%%%%%%%%%
\begin{abstract}
\noindent
We discuss representations of the projective line over a ring $R$
with $1$ in a projective space over some (not necessarily
commutative) field $K$. Such a representation is based upon a
$(K,R)$-bimodule $U$.
The points of the projective line over $R$ are represented
by certain subspaces of the projective space $\PP(K,U\times U)$ that
are isomorphic to one of their complements. In particular, distant
points go over to complementary subspaces, but in certain cases, also
non-distant points may have complementary images.

\end{abstract}
%%%%%%%%%%%%%%%%%%%%%%%%%%%%%%%%%
\parskip1mm
\parindent0cm
%%%%%%%%%%%%%%%%%%%%%%%%%%%%%%%%%%%%%%%%%%%%%%%%%%%%%%%%%%%%%%%%%%%
\section{Introduction}

Each ring $R$ with $1$, containing in its centre a (necessarily
commutative) field $F$ with $1\in F$, gives rise to a {\em chain
geometry} $\Sigma(F,R)$. For a survey, see \cite{herz-95}. In
\cite{blu+h-99a} we  introduced the concept of a
{\em generalized chain geometry}
$\Sigma(F,R)$; now $R$ is a ring with $1$
containing a (not necessarily commutative) field $F$ subject to $1\in
F$. In both cases the point set of $\Sigma(F,R)$ is $\PP(R)$, i.e.,
the {\em projective line} over $R$, and the chains are the
$F$-sublines.

In the present paper we introduce representations of the projective
line over an arbitrary ring $R$ in a projective space over some field
$K$. In a second publication our results will be applied to obtain
representations of generalized chain geometries.

The starting point of our investigation is A.\ Herzer's approach
\cite{herz-95} to obtain a model of a chain geometry $\Sigma(F,R)$
for a finite-dimensional $F$-algebra $R$ by means of a faithful right
$R$-module $U$ with finite $F$-dimension, say $r$. Here the points of
the projective line $\PP(R)$ are in one-one correspondence with
certain $(r-1)$-dimensional subspaces of the $(2r-1)$-dimensional
projective space $\PP(F,U\times U)$. In our more general setting we
use a $(K,R)$-bimodule $U$; so $U$ is a left $K$-vector space and at
the same time a right $R$-module. We neither do assume that $K$
is a subset of $R$, nor that $U$ is a faithful $R$-module, nor that
the $K$-dimension of $U$ is finite. A {\em projective
representation}\/ obtained in this way maps the points of $\PP(R)$
into the set of those subspaces of the projective space
$\PP(K,U\times U)$ that are isomorphic to one of their complements.
This mapping is injective if, and only if, $U$ is a faithful
$R$-module. In this case we speak of a {\em projective model}\/ of
$\PP(R)$.

If $U'$ is a sub-bimodule of $U$ then there are representations of
$\PP(R)$ that stem from the action of $R$ on $U'$ and $U/U'$. In
general, these {\em induced representations} are not injective, even
if $U$ is faithful. This is one of the reasons why we also discuss
non-injective representations. The examples at the end of the paper
illustrate how these induced representations can sometimes be used in
order to describe models of $\PP(R)$ in terms of $\PP(K,U\times U)$.
% ----------

\section{The projective line over a ring}
%%%%%%%%%%%%%%%%%%%%%%%%%%

Let $R$ be a ring. Throughout this paper we shall only consider
rings with $1$ (where the trivial case $1=0$ is not excluded).
The group of
invertible elements of the ring $R$ will be denoted by $R^*$.
Consider the free left $R$-module $R^2$.
Its automorphism group is the group $\GL_2(R)$ of invertible
$2\times 2$-matrices with entries in $R$.
According to \cite{blu+h-99a}, \cite{herz-95},  the {\em projective line
over $R$} is the orbit
$$\PP(R):=R(1,0)^{\GL_2(R)}$$
of the free cyclic submodule $R(1,0)$ under the action of $\GL_2(R)$.
Since $R^2=R(1,0)\oplus R(0,1)$, the elements (the {\em points})
of $\PP(R)$ are exactly
those free cyclic submodules of $R^2$ that have a free cyclic
complement.

A pair $(a,b)\in R^2$ is called
 {\em admissible}, if there exist $c,d\in R$ such that
$\SMat2{a&b\\c&d}\in \GL_2(R)$.
So $\PP(R)=\{R(a,b)\subset R^2\mid (a,b)$ admissible$\}$.
However, in certain cases the  points of $\PP(R)$ may also
be represented by  non-admissible
pairs, as we will see below.

We recall that a pair $(a,b)\in R^2$
is {\em unimodular}, if there exist $x,y\in R$ such that $ax+by=1$, i.e.,
if there is an $R$-linear
form $R^2\to R$ mapping $(a,b)$ to $1$. This is equivalent to saying that
the right ideal generated by $a$ and $b$ is the whole ring~$R$.

Obviously, each admissible
pair $(a,b)$ is unimodular.
If $R$ is commutative, then
admissibility and unimodularity are equivalent.
W.\ Benz in \cite{benz-73} considers only commutative rings and
defines the projective line using unimodular pairs.

\begin{prop}\label{LRinvert}
Let $(a,b)\in R^2$ be admissible, and let $s\in R$. Put
$(a',b'):=s(a,b)$. Then
\begin{enumerate}
\item $s$ is left invertible $\Longleftrightarrow$ $R(a,b)=R(a',b')$.
\item $s$ is right invertible $\Longleftrightarrow$ $(a',b')$ is admissible.
\end{enumerate}
\end{prop}
\begin{proof}
(1): If there is an $l\in R$ with $ls=1$, then $(a,b)=l(a',b')$. So
$R(a,b)=R(a',b')$.\\
If $R(a,b)=R(a',b')$, then there is an $l\in R$ such that $(a,b)=l(a',b')$.
Since $(a,b)$ is admissible, it is also unimodular, and so there are $x,y\in R$
with $1=ax+by=lsax+lsby=ls$. Hence $s$ is left invertible.

(2): If $s$ is right invertible, then $sr=1$ for some $r\in R$. An easy
calculation shows that
\begin{equation}\label{matrix}\gamma=\Mat2{s&0\\1-rs&r}\in\GL_2(R), \
{\rm with} \
\gamma^{-1}=\Mat2{r&1-rs\\0&s}.
\end{equation}
There is a matrix $\SMat2{a&b\\c&d}\in\GL_2(R)$,
whence $\SMat2{a'&b'\\{*}\ &*\ }= \gamma\SMat2{a&b\\c&d}\in\GL_2(R)$,
as required.\\
If $(a',b')$ is admissible, then there are $x',y'\in R$ with $a'x'+b'y'=1$.
So $s(ax'+by')=1$, i.e., $s$ has a right inverse.
\end{proof}

Note that the statement of Proposition \ref{LRinvert}
remains true if one substitutes ``admissible''
by ``unimodular'', however, the proof of (2)``$\Rightarrow$''
then has to be modified.

Rings with the property that $ab=1$ implies $ba=1$ are called
{\em Dedekind-finite} (see e.g. \cite{lam-91}). From
Proposition~\ref{LRinvert} we obtain

\begin{prop}\label{Dedekind}
Let $R$ be a ring. Then the following are equivalent:
\begin{enumerate}
\item $R$ is Dedekind-finite.
\item If $R(a,b)\in\PP(R)$, then $(a,b)$ is admissible.
\item No point of $\PP(R)$ is properly contained in another
point of $\PP(R)$.
\end{enumerate}
\end{prop}

\begin{rem}\label{contained}
If $R$ is not Dedekind-finite, then each point
$p\in\PP(R)$ belongs to an infinite sequence of
points
$$\ldots\subsetneq p_{-2}\subsetneq p_{-1}\subsetneq p_0=p
\subsetneq p_1\subsetneq p_2\subsetneq\ldots$$
Namely, let $\gamma$ be the matrix of formula (\ref{matrix}),
where $sr=1\ne rs$.
Then Proposition
\ref{LRinvert} shows that the points $p_i:=p^{\gamma^i}$ are as desired.
\end{rem}

Recall that according to F.D.\ Veldkamp \cite{veld-85}, \cite{veld-95}
the ring $R$ has {\em stable rank~$2$}, if for each unimodular pair
$(a,b)\in R^2$ there is a  $c\in R$ such that $a+bc$ is right invertible.
The following results on rings of stable rank~$2$ can be found in
\cite{veld-85} (results 2.10 and 2.11):

\begin{rem}\label{stable}
Let $R$ be of stable rank~$2$. Then
 $R$ is Dedekind-finite and each unimodular $(a,b)\in R^2$ is admissible.
\end{rem}

Note that  Herzer's definition of stable rank~$2$ in \cite{herz-95}
seems to be stronger but actually coincides with Veldkamp's because of
\ref{stable}. Moreover, it is not necessary to distinguish between
left and right stable rank $2$ because by \cite{veld-85}, 2.2,
the opposite ring (with reversed multiplication) of a ring of stable
rank~$2$ also has stable rank~$2$.

Using results of \cite{lam-91}, \S\ 20, one obtains that each
left (or right) artinian ring has stable rank~$2$ (called ``left stable
range~$1$'' in \cite{lam-91}). We shall need the following special case:

\begin{rem}\label{finDim}
Assume that $R$ contains a subfield $K$ such that $R$ is a
{\em finite-dimensional} left (or right) vector space over $K$.
Then $R$ is of stable rank~$2$. In particular, $R$ is Dedekind-finite.
\end{rem}

Here by a {\em subfield} we mean a not necessarily commutative field
$K\subset R$ with $1\in K$.

%%%%%%%%%%%%%%%%%%%%%%%%%%%%%%%%%%%%%%%%%%%%%%%%%
%\section{Distance}
%%%%%%%%%%%%%%%%%%%%%%%%%%%%%%%%%%%%%%%%%%%%%%%%

We turn back to the projective line over an arbitrary ring.
The point set $\PP(R)$ is endowed with the symmetric
relation $\dis$ ({\em ``distant''}) defined by
$$\dis:=\{R(1,0), R(0,1)\}^{\GL_2(R)}$$
i.e., two points $p,q\in \PP(R)$ are distant
exactly if there is a $ \gamma\in\GL_2(R)$ mapping  $R(1,0)$
to  $p$
and  $R(0,1)$ to  $q$.
Distance can also be expressed  in terms of coordinates:
\begin{rem}
Let $p=R(a,b)$, $q= R(c,d)\in \PP(R)$ with admissible $(a,b)$, $(c,d)$.
Then $$p\dis q\iff \Mat2{a&b\\c&d}\in \GL_2(R).$$
\end{rem}
Note that this is independent of the choice of the {\em admissible}
representatives $(a,b)$, $(c,d)$. In addition,
$\dis$ is anti-reflexive exactly if $1\ne 0$; compare \cite{herz-95}.

By definition, two points of $\PP(R)$ are distant if, and only if,
they are complementary submodules of
$R^2$. There are several possibilities for points being non-distant,
which all can occur as the following examples show:

\begin{exas}\label{nondist}
\begin{enumerate}
\item Let $R$ be a ring that is not Dedekind-finite. Let
$\gamma\in\GL_2(R)$ be defined as in Remark \ref{contained}. Then
$p=R(1,0)^\gamma=R(s,0)$ and $q=R(0,1)$ are non-distant: They
have a trivial intersection but they do not span $R^2$. \\
Now consider $p'=R(1,0)^{\gamma^{-1}}=R(r,1-rs)$ (see formula (\ref{matrix})).
Then $p'$ and $q$ are non-distant: They span $R^2$, but
$(1-rs)(r,1-rs)=(0,1-rs)\ne (0,0)$ lies in their intersection.
\item
Let $R$ contain a subfield $K$ such that $R$, considered as
left vector space over $K$, has finite dimension $n$. Then all
points of $\PP(R)$ are $n$-dimensional subspaces of the left vector
space $R^2$. In particular, two points have a trivial intersection
exactly if they span $R^2$.
\end{enumerate}
\end{exas}

In Example \ref{poly} below
we will see an example of a commutative (and hence Dedekind-finite)
ring where non-distant points intersect trivially.
\section{Homomorphisms}

Now we want to study mappings between projective lines over rings
that are induced by ring homomorphisms.

From now on, we will follow the convention that whenever a
point of $\PP(R)$ is given in the form  $R(a,b)$,
we always assume that the pair $(a,b)\in R^2$ is admissible.

Let $R,S$ be rings.
The distance relations on $\PP(R)$ and $\PP(S)$ are denoted by $\dis_R$
and $\dis_S$, respectively.
Consider a ring homomorphism $\phi:R\to S$,
where we always  suppose that $1_R$ is mapped to $1_S$.
Associated to $\phi$ is a homomorphism
$\M(2\times2,R)\to \M(2\times2,S)$, mapping $\SMat2{a&b\\c&d}$ to
$\SMat2{a^\phi&b^\phi\\c^\phi&d^\phi}$, which will also be denoted by $\phi$.
Its restriction to $\GL_2(R)$ is a group homomorphism into $\GL_2(S)$.
This implies that if $(a,b)\in R^2$ is admissible, so is $(a^\phi,b^\phi)
\in S^2$, and we can introduce the mapping
$$\bar\phi:\PP(R)\to \PP(S):R(a,b)\mapsto S(a^\phi,b^\phi).$$

\begin{prop}\label{homo}
Let $\phi:R\to S$ be a ring homomorphism. Then for
$\bar\phi:\PP(R)\to\PP(S)$
the following statements hold:
\begin{enumerate}
\item $\bar\phi$ {\em preserves distance}, i.e.,
 $\forall p,q\in\PP(R):\ p\dis_R q\Rightarrow
 p^{\bar\phi} \dis_S q^{\bar\phi}$.
\item $\bar\phi$ is {\em compatible with the action of $\GL_2(R)$},
i.e. $\forall p\in\PP(R)\
\forall \gamma\in\GL_2(R): \ p^{\gamma\bar\phi}=p^{\bar\phi\gamma^\phi}$.
\item $\bar\phi$ is injective if, and only if, $\phi$ is.
\end{enumerate}
\end{prop}
\begin{proof} Only (3) deserves our attention. Let $\phi$
be injective.  Assume  that
$R(a,b)^{\bar\phi}=R(c,d)^{\bar\phi}$ holds
for $R(a,b)$, $R(c,d)\in\PP(R)$. Then
there is an $s\in S^*$ with $(a^\phi,b^\phi)=s(c^\phi,d^\phi)$.
Since $(c,d)\in R^2$ is
unimodular, there are $x,y\in R$ with $s=s1=s1^\phi=s(cx+dy)^\phi=
a^\phi x^\phi+b^\phi y^\phi\in R^\phi$.
Analogously, one sees that $s^{-1}\in R^\phi$. Hence $s\in (R^\phi)^*$,
which equals $(R^*)^\phi$, since $\phi$ is injective. So $R(a,b)=R(c,d)$.\\
Now let $\bar\phi$ be injective, and assume $a^\phi=b^\phi$ for $a,b\in R$.
Then $R(1,a)^{\bar\phi}=S(1,a^\phi)=S(1,b^\phi)=R(1,b)^{\bar\phi}$,
whence $R(1,a)=R(1,b)$ and so $a=b$.
\end{proof}

We call the mapping $\bar\phi:\PP(R)\to\PP(S)$ the {\em homomorphism
of projective lines induced by $\phi:R\to S$}.
Such homomorphisms map distant points to distant points.
However,  they may also map non-distant points to distant points: Consider
e.g. the homomorphism $\PP(\ZZ)\to\PP(\QQ)$ induced by the natural
inclusion $\ZZ\to\QQ$. This injective homomorphism actually is a bijection,
since each element of $\PP(\QQ)$ can be represented by a pair of
relatively prime integers.
The points $\ZZ(1,0)$ and
$\ZZ(1,2)$ are non-distant because $\SMat2{1&0\\1&2}$ is not
invertible over~$\ZZ$. However, their image points $\QQ(1,0)$
and $\QQ(1,2)$ are different and hence distant in $\PP(\QQ)$.

The following gives a characterization of the homomorphisms $\bar\phi$
that
preserve also non-distance. By $\rad(R)$ we denote the
(Jacobson) {\em radical}
of the ring $R$ (cf.\ \cite{lam-91}).

\begin{prop}\label{nondistEmb}
Let $\bar\phi:\PP(R)\to\PP(S)$ be  induced by the
ring homomorphism $\phi:R\to S$. Then the  following statements are
equivalent:
\begin{enumerate}
\item $\forall p,q\in\PP(R):\ p^{\bar\phi}\dis_S q^{\bar\phi}
\Rightarrow p\dis_R q$.
\item $\forall y\in R: y^\phi\in S^*\Rightarrow y\in R^*$.
\item $\ker(\phi)\subset\rad(R)$ and $(R^\phi)^*=S^*\cap R^\phi$.
\end{enumerate}
\end{prop}
\begin{proof}
(1) $\Rightarrow$ (2): For $r\in R$ with $r^\phi\in S^*$ we have
$S(1,0)\dis_S S(1,r^\phi)$. Hence  condition
(1) implies $R(1,0)\dis_R R(1,r)$
and thus $r\in R^*$.\\
(2) $\Rightarrow$ (1):
Let $p^{\bar\phi}\dis_S q^{\bar\phi}$ hold for $p,q\in \PP(R)$.
Choose $\gamma\in\GL_2(R)$ with $p^\gamma=R(1,0)$. Then $q^\gamma=R(x,y)$
for a certain admissible pair $(x,y)\in R^2$. By \ref{homo}(2), we have
$S(1,0)=p^{\gamma{\bar\phi}}=p^{{\bar\phi}\gamma^\phi}
\dis_S q^{{\bar\phi}\gamma^\phi}=q^{\gamma{\bar\phi}}=
S(x^\phi,y^\phi)$, and hence  $y^\phi\in S^*$.  So, by (2),
$y\in R^*$. This implies $p^\gamma\dis_R q^\gamma$ and thus
also $p\dis_R q$.\\
(2) $\Leftrightarrow$ (3):
See \cite{fer+v-85}, Lemma 1.5.
\end{proof}

As the example $\PP(\ZZ)\to\PP(\QQ)$
above shows, the ring homomorphism $\phi$ need not be surjective
if $\bar\phi$ is.

We now consider the case where $\phi:R\to S$
is a surjective homomorphism of rings. It is not clear whether
in general
$\bar\phi$ also is surjective. We study special cases.

According to J.R. Silvester \cite{silv-81} we introduce the following
notions for a ring $R$:

The {\em elementary linear group} $\E_2(R)$ is the subgroup of
$\GL_2(R)$ generated by the {\em elementary transvections}, i.e,
by all matrices of the form $\SMat2{1&0\\x&1}$ or $\SMat2{1&x\\0&1}$
($x\in R$).
The group $\GE_2(R)$ is the subgroup of $\GL_2(R)$
generated by $\E_2(R)$ and all diagonal matrices $\SMat2{a&0\\0&b}\in
\GL_2(R)$. Note that $\E_2(R)$ is normal in $\GE_2(R)$.
If $\GE_2(R)=\GL_2(R)$, then $R$ is called a {\em $\GE_2$-ring}.

Examples of $\GE_2$-rings and also of rings that are not
$\GE_2$-rings can be found in \cite{silv-81},\ p.114 and p.121,
respectively. Important
for us is the following:

\begin{rem}\label{SR2GE2}(See \cite{hahn+om-89}, 4.2.5.)
Let $R$ be a ring of stable rank $2$. Then $R$ is a $\GE_2$-ring.
\end{rem}

\begin{lem}\label{orbitE2}
Let $R$ be a $\GE_2$-ring. Then $\PP(R)=R(1,0)^{\E_2(R)}$.
\end{lem}
\begin{proof}
Let $p=R(1,0)^\gamma\in\PP(R)$, with $\gamma\in\GL_2(R)$.
Since $\GL_2(R)=\GE_2(R)$ and $\E_2(R)$ is normal in $\GE_2(R)$,
we have $\gamma=\delta\eta$, where $\delta=\SMat2{a&0\\0&b}$
and $\eta\in\E_2(R)$. So $p=R(1,0)^\gamma=R(a^{-1},0)^\gamma
=R(1,0)^\eta\in R(1,0)^{\E_2(R)}$.
\end{proof}

Now we can state conditions that imply that with $\phi:R\to S$ also
$\bar\phi$ is surjective.

\begin{prop}\label{surject}
Let $\phi:R\to S$ be a surjective homomorphism of rings.
Then also $\bar\phi:\PP(R)\to\PP(S)$ is surjective, if one
of the following conditions is satisfied:
\begin{enumerate}
\item $S$ is a $\GE_2$-ring.
\item $\ker(\phi)\subset\rad(R)$.
\item $R$ is the internal direct product of $\ker(\phi)$ and some
ideal $R'\subset R$.
\end{enumerate}
\end{prop}
\begin{proof}
(1): Consider a point $q\in\PP(S)$. By Lemma \ref{orbitE2} we have
$q=S(1,0)^\eta$, where $\eta\in\E_2(S)$, i.e., $\eta$ is a product
of elementary transvections. Since $\phi:R\to S$ is surjective,
each elementary transvection has a preimage under
$\phi:\M(2\times 2,R)\to\M(2\times 2,S)$ that is an elementary
transvection over $R$. Hence $\eta=\gamma^\phi$, where $\gamma\in\E_2(R)$,
and so by \ref{homo}(2) we obtain
$q=R(1,0)^{\bar\phi\eta}=R(1,0)^{\gamma\bar\phi}\in
\PP(R)^{\bar\phi}$.\\
(2): Follows from \cite{blu+s-95}, Lemma 1.14.\\
(3): In this case, $\GL_2(R)$ consists exactly of the sums
$\gamma+\gamma'$, where $\gamma\in\GL_2(\ker(\phi))$
and $\gamma'\in\GL_2(R')$. Moreover, $\phi|_{\GL_2(R')}:\GL_2(R')\to
\GL_2(S)$ is an isomorphism of groups. This yields the assertion.
\end{proof}

Note that one could also use Proposition \ref{nondistEmb}
in order to prove assertion (2) above, since the radical of $\M(2\times 2,R)$
consists exactly of all matrices with entries in $\rad(R)$.

%%%%%%%%%%%%%%%%%%%%%%%%%%%%%%%%%%%%%%%%%%%%%%%%%%%%%
\section{Projective representations}
%%%%%%%%%%%%%%%%%%%%%%%%%%%%%%%%%%%%%%%%%%%%%%%%%%%%

The projective representations we are now aiming at
are based upon the following.

\begin{rem}\label{isoProj}(see \cite{blunck-99}).
Let $U$ be a left vector space over a field $K$, and let $S=\End_K(U)$ be its
endomorphism ring. Moreover, let $\cG$ be the set of all subspaces of the
projective space $\PP(K,U\times U)$ that are isomorphic to one of their
complements. Then $$\Psi:\PP(S)\to \cG: S(\alpha,\beta)\mapsto
U^{(\alpha,\beta)}:= \{(u^\alpha,u^\beta)\mid u\in U\}$$ is a well-defined
bijection mapping distant points of $\PP(S)$ to complementary subspaces in
$\cG$ and non-distant points to non-complementary subspaces. Moreover, the
groups $\GL_2(S)$ and $\Aut_K(U\times U)$ are isomorphic and their actions on
$\PP(S)$ and on $\cG$, respectively, are equivalent via $\Psi$. In particular,
the mappings induced on $\cG$ by $\GL_2(S)$ arise from {\em projective
collineations} of the projective space $\PP(K,U\times U)$.
\end{rem}

Let now   $K$ be a field  and let
 $R$ be a ring. A left vector space $U$ over $K$ is called   a
{\em $(K,R)$-bimodule}, if $U$ is  a (unitary) right $R$-module
such that for all $k\in K$, $u\in U$, $a\in R$ the equality
$k(u\cdot a)=(ku)\cdot a$ holds. If $U$ is a $(K,R)$-bimodule,
then $\phi:R\to \End_K(U)$ with $a^\phi=\rho_a:u\mapsto u\cdot a$ is a
ring homomorphism.

If, on the other hand,
there is a homomorphism $\phi:R\to \End_K(U)$, then $U$ becomes a
$(K,R)$-bimodule by setting $u\cdot a:=u^{\rho_a}$, where $\rho_a=a^\phi$.
A homomorphism $\phi:R\to \End_K(U)$ is also called a
{\em $K$-linear representation}
of $R$.

So, the concepts of a $K$-linear representation of $R$ and
a $(K,R)$-bimodule are equivalent. Whenever we consider a $(K,R)$-bimodule
$U$, we denote by $\phi$ the associated linear representation, and
for $a\in R$ we write $\rho_a$ for the endomorphism $a^\phi:u\mapsto u\cdot a$.

A $(K,R)$-bimodule $U$ and  the associated linear representation
$\phi$ are called {\em faithful}, exactly if $\phi$ is an injection.

Combining
\ref{homo} and \ref{isoProj}, we obtain our main result:

\begin{thm}\label{projRep}
Let $U$ be a $(K,R)$-bimodule. Then the mapping
$$\Phi:=\bar\phi\Psi:\PP(R)\to \cG:R(a,b)\mapsto %U^{(a,b)}:=
U^{(\rho_a,\rho_b)}$$
 maps distant points of $\PP(R)$
to complementary subspaces in $\PP(K,U\times U)$.
The bimodule  $U$ is faithful if, and only  if, $\Phi$ is injective.
\end{thm}

Thus, to each homomorphism $\phi:R\to\End_K(U)$ corresponds a mapping $\Phi$
(see above).  We call $\Phi$ a {\em projective representation} of $\PP(R)$, and
a {\em faithful} projective representation if $U$ is faithful. We are
interested in the image of $\PP(R)$ under a projective representation. If the
representation is faithful, then $\Phi:\PP(R)\to\PP(R)^{\Phi}$ is a bijection,
and the image $\PP(R)^{\Phi}$ can be seen as a model of $\PP(R)$ in the
projective space; we then call $\PP(R)^{\Phi}$ a {\em projective model} of
$\PP(R)$. Otherwise, one obtains a model of the projective line over another
ring:

\begin{prop}\label{factor}
Let $J=\ann(U)$ be the {\em annihilator} of $U$, i.e.,
the kernel of the representation $\phi:R\to\End_K(U)$.
Then the following statements
hold:
\begin{enumerate}
\item The mapping $\phi_f:R/J\to \End_K(U)$
with $\rho_{a+J}:u\mapsto u^{\rho_a}$ is a faithful $K$-linear representation
of~$R/J$. Hence $\Phi_f={\overline{\phi_f}}\Psi$ is a faithful projective
representation of $\PP(R/J)$.
\item The projective model $\PP(R/J)^{\Phi_f}$ contains $\PP(R)^\Phi$.
\item The mapping $\bar\pi:\PP(R)\to\PP(R/J)$ induced by the canonical
epimorphism $\pi:R\to R/J$ is surjective if, and only if,
$\PP(R/J)^{\Phi_f}=\PP(R)^\Phi$.
\end{enumerate}
\end{prop}

Recall that Proposition \ref{surject} gives conditions under which
the assumptions of statement (3) are met.

The representation $\phi:R\to S=\End_K(U)$ gives rise to a group homomorphism
$\phi:\GL_2(R)\to \GL_2(S)\cong\Aut_K(U\times U)$.   Using \ref{homo}(2) and
\ref{isoProj} we obtain  the following:

\begin{prop}\label{action}
Let $U$ be a $(K,R)$-bimodule, and let $\gamma\in\GL_2(R)$. Then the  induced
mapping $$\PP(R)^\Phi\to\PP(R)^\Phi:R(a,b)^\Phi\mapsto R(a,b)^{\gamma\Phi}$$ is
induced by a projective collineation of $\PP(K,U\times U)$.
\end{prop}

Finally,  Proposition
\ref{nondistEmb} yields the following:

\begin{prop}\label{nonComp}
Let $U$ be a  $(K,R)$-bimodule. Then the corresponding projective
representation $\Phi$ maps non-distant points to non-complementary subspaces
exactly if for each $a\in R$ the condition $\rho_a\in\Aut_K(U)$ implies $a\in
R^*$.
\end{prop}

Note that from $\rho_a\in\Aut_K(U)$ and $a\in R^*$ one obtains
that $(\rho_a)^{-1}
=\rho_{a^{-1}}$.

We mention two classes of examples where the condition of
Proposition \ref{nonComp}
is satisfied.

\begin{exas}
\begin{enumerate}
\item Let $R$ contain $K$ as a subfield. Then $U=R$ is a
left vector space over $K$, and
$\phi:R\to \End_K(U)$ with $\rho_a:x\mapsto xa$ is a faithful linear
representation of $R$, called the {\em regular representation}.
In this case $\Phi$ is the
identity, where the submodule $R(a,b)\in\PP(R)$ is considered as
a projective subspace of $\PP(K,U\times U)$. So points of $\PP(R)$
are distant exactly if their $\Phi$-images are complementary.
This reflects the algebraic fact that the endomorphism
$\rho_a:R\to R: x\mapsto xa$ is a bijection exactly if $a\in R^*$.
\item Let $U$ be a faithful $(K,R)$-bimodule. Assume moreover that $R$
contains a subfield $L$ such that $R$ is a finite-dimensional left
vector space over $L$. Then the projective representation
$\Phi$ maps non-distant points to non-complementary
subspaces:\\
In view of (1), it suffices to show that for
each $a\in R$ with $\rho_a\in\Aut_K(U)$ the
$L$-linear mapping $R\to R:x\mapsto xa$ is injective.
Suppose $xa=0$ for $x\in R$. Then for all $u\in U$ we have
$0=u\cdot0=u\cdot(xa)=(u\cdot x)^{\rho_a}$. Since $\rho_a$ is an automorphism, this
implies $u\cdot x=0$ for all $u\in U$, and hence $x=0$ because $U$ is a
faithful $R$-module.
\end{enumerate}
\end{exas}

We proceed by giving an example of a
faithful projective representation where non-distant points
appear as complementary subspaces:

\begin{exa}\label{poly}
Let $K$ be any commutative field, let
$R$  be the polynomial ring $R=K[X]$, and let $U=K(X)$ be its field
of fractions. Then $U$ contains $K$ and $R$, and thus is a faithful
$(K,R)$-bimodule in a natural way.
Obviously, $\rho_X:u\mapsto uX$ is a bijection on $U$, but $X\not\in R^*$.

This means that e.g.
$R(1,0)$ and $R(1,X)$ are non-distant points of $\PP(R)$,
but their images $U^{(1,0)}=U\times \{0\}$ and $U^{(1,\rho_X)}=
\{(u,uX)\mid u\in U\}$ are complementary subspaces of $\PP(K,U\times U)$.
Note that $R(1,0)$ and $R(1,X)$, considered as submodules of $R^2$,
also intersect trivially, but they do not span $R^2$ (compare
\ref{nondist}).

Note, moreover,
that  we could also interpret the elements $R(a,b)^{\Phi}
=U(a,b)$
as points of the projective line over the field $U$. Hence any two
such elements must be complementary.

In a similar way one can also construct examples where
$R$ is not contained in any field:   Let $R$ and $U$ be as above.
Let $R[\eps]$ be the ring of {\em dual numbers} over $R$,
 with $\eps$  central, $\eps\not\in K$,  and $\eps^2=0$. Then $\eps $ is
a zero-divisor and hence $R[\eps]$ is not embeddable into any field.
Now take $U[\eps]$ and proceed as above.
\end{exa}

%%%%%%%%%%%%%%%%%%%%%%%%%%%%%%%%%%%%%%%%%%%%%%%%%%%%%%%%%%
%\section{Special Constructions}
%%%%%%%%%%%%%%%%%%%%%%%%%%%%%%%%%%%%%%%%%%%%%%%%%%%%%%%%%%%%%

Let $U$ be a $(K,R)$-bimodule.
A subset $U'\subset U$ is called
a {\em sub-bimodule} of $U$, if $U'$ is a subspace of the left
vector space $U$ over $K$ and at the same time a submodule of
the right $R$-module~$U$. The linear representation of $R$ given by
the bimodule $U'$ is $\phi':a\mapsto{\rho_a}|_{U'}$.
The faithful representation
$(\phi')_f:R/\ann(U')\to \End_K(U')$ will be  called the {\em induced
faithful representation}.

The projective representation $\Phi'$ associated to $\phi'$
maps the points of $\PP(R)$ to certain subspaces
of the projective space $\PP(K,U'\times U')$, more exactly,
$\PP(R)^{\Phi'}$ is a subset of the  set $\cG'$
of all subspaces of $\PP(K,U'\times U')$ that are isomorphic to
one of their complements.

Now $\PP(K,U'\times U')$ is a projective subspace
of $\PP(K,U\times U)$, and we can compare the images of $\PP(R)$ under
the projective representations $\Phi$ and $\Phi'$. One
obtains the following geometric interpretation:

\begin{prop}\label{schlauch}
Let $U'$ be a sub-bimodule of the $(K,R)$-bimodule $U$, and let
$\Phi'$ and $\Phi$ be the associated projective representations
of $\PP(R)$. Then for each $p\in\PP(R)$ we have
$$p^{\Phi'}=p^{\Phi}\cap (U'\times U').$$
In particular, each $p^{\Phi}$ meets the projective
subspace $\PP(K,U'\times U')$ in an element of $\cG'$.
\end{prop}
\begin{proof}
First consider $p=R(1,0)$. Then
$p^{\Phi'}=U'\times\{0\}=
(U\times \{0\})\cap (U'\times U')=p^{\Phi}\cap(U'\times U')$.
Now consider an arbitrary $p\in\PP(R)$. Then $p=R(1,0)^\gamma$ for
some $\gamma\in\GL_2(R)$. The induced automorphism $\gamma^\phi$
of $U\times U$ leaves $U'\times U'$ invariant, it coincides on
$U'\times U'$ with $\gamma^{\phi'}\in\Aut_K(U'\times U')$. This
yields the assertion.
\end{proof}

Note that the ${\Phi'}$-image of  $\PP(R)$ is contained
in the image of  $\PP(R/\ann(U'))$ under the induced faithful representation
$(\Phi')_f$. According to \ref{factor}(3),
the two sets  coincide exactly if
the mapping $\bar\pi:\PP(R)\to
\PP(R/\ann(U'))$,
associated to the canonical epimorphism $\pi:R\to R/\ann(U')$,
is surjective.

\begin{prop}\label{dirSumA}
Let $U=U'\oplus U''$ be a  $(K,R)$-bimodule. Let $\phi$, $\phi'$, $\phi''$ be
the associated representations of $R$. Then for each    $p\in\PP(R)$ we have
$p^{\Phi}=p^{\Phi'}\oplus p^{\Phi''}$.
\end{prop}
\begin{proof}
As in the proof of Proposition \ref{schlauch} we first verify
the assertion for $p=R(1,0)$ (with the help of \ref{schlauch})
and then use the action of $\GL_2(R)$.
\end{proof}

Let again $U'$ be a sub-bimodule of the $(K,R)$-bimodule $U$.
Then also $\widetilde U=U/U'$ is a $(K,R)$-bimodule, corresponding to
the representation $\widetilde\phi:R\to\End_K(\widetilde U)$,
where $\widetilde\rho_a:u+U'\mapsto u^{\rho_a}+U'$.
The kernel of this representation is the ideal
consisting of all $a\in R$ such that the image of $\rho_a$
is contained in $U'$. As above, we obtain an
{\em induced faithful representation} $(\widetilde\phi)_f:R/\ker(\widetilde\phi)
\to\End_K(\widetilde U)$.

The  projective representation $\widetilde\Phi$
maps $\PP(R)$ into the set $\widetilde \cG$ of all
subspaces of $\PP(K,\widetilde U\times
\widetilde U)$ that are isomorphic to one of their complements.
Now the projective space
$\PP(K,\widetilde U\times\widetilde U)$ is canonically isomorphic to
the projective space of all subspaces of  $\PP(K,U\times U)$
containing $U'\times U'$, because $(U\times U)/(U'\times U')
\cong \widetilde U\times \widetilde U$. We shall identify the elements of
$\widetilde\cG$ with their images under this isomorphism.
So we can compare
$\widetilde\Phi$ and $\Phi$, and the same procedure as before yields

\begin{prop}\label{ProjFactor}
Let $\widetilde U=U/U'$, and let $\widetilde\Phi$ be the associated projective
representation of $\PP(R)$.
Then for each $p\in\PP(R)$ we have
$$p^{\widetilde\Phi}=p^{\Phi}+ (U'\times U').$$
In particular, each $p^{\Phi}+ (U'\times U') $
is an element of $\widetilde\cG$.
\end{prop}

As before, one may also consider the induced faithful representation
$(\widetilde\Phi)_f$ of $\PP(R/\ker(\widetilde\phi))$.

%%%%%%%%%%%%%%%%%%%%%%%%%%%%%%%%%%%%%%%%%%%%%%%%%%%%%%%%%%%%%%%%%%%%
\section{Examples}

In this section we study some examples. Note that we consider
only  rings $R$ that are
finite-dimensional left vector spaces over a subfield $K$. Then for
each ideal $I$ of $R$ also the ring
$R/I$ is finite dimensional over $K$, whence  $R/I$ is of stable rank $2$
and hence a $\GE_2$-ring (compare \ref{finDim} and \ref{SR2GE2}).
So Proposition \ref{surject} implies that
in all our examples the mapping $\bar\pi:\PP(R)\to\PP(R/I)$
induced by the canonical epimorphism $\pi:R\to R/I$ is surjective.

\begin{exa}\label{reguli}
Let $K=R$ be any (not necessarily commutative) field and let
$U=K^2$ with componentwise action $(x_1,x_2)\cdot k=(x_1k,x_2k)$.
Then
$U$ is the direct sum of the sub-bimodules $U_1=K(1,0)$ and $U_2=K(0,1)$,
on which $R=K$ acts faithfully in the natural way.
The  representations induced in the skew lines
$U_i\times U_i$ are faithful and map $\PP(K)$
onto the set of all points of $U_i\times U_i$.
Moreover, $\beta:=\Phi_1^{-1}\Phi_2$ is a
bijection
between these two projective lines, which is linearly induced and hence
a  projectivity.
The elements of the projective model $\PP(K)^{\Phi}$
in $\PP(K,U\times U)$ are exactly the  lines joining a point of
$U_1\times U_1$ and its $\beta$-image in $U_2\times U_2$. So
$\PP(K)^{\Phi}$ is a regulus
in $3$-space (compare \cite{brau-76}).

The same applies if $U=K^n$. Then one obtains a regulus in a
$(2n-1)$-dimensional projective space (see \cite{blunck-00a}),
i.e., a generalization to the not necessarily pappian case
of
a family of $(n-1)$-dimensional subspaces
on a Segre manifold $S_{n-1,1}$ (compare \cite{bura-61}).
\end{exa}

\begin{exa}
Example \ref{reguli} above can be modified in the following way: Let
$\alpha_1,\alpha_2:K\to K$ be field monomorphisms. Then $K$ acts
faithfully on $U=K^2$ via
$(x_1,x_2)\cdot k=(x_1k^{\alpha_1},x_2k^{\alpha_2})$.
The induced projective models of $\PP(K)$ in the projective lines
$U_i\times U_i$
are projective sublines over the subfields $K^{\alpha_i}$.
In general, the bijection $\beta$ between the two models
is not $K$-semilinearly induced.

We mention one  special case: If $K=\CC$, $\alpha_1=\id$, and $\alpha_2$
is the complex conjugation,
then the projective model of
$\PP(\CC)$ is a set of lines in the $3$-space $\PP(\CC,U\times U)$.
It can be interpreted as follows: The $\alpha_2$-semilinear bijection
$\beta$ extends to a collineation of order two which fixes a
Baer subspace (with $\RR$ as underlying field). The lines of the
projective model of $\PP(\CC)$ meet this Baer subspace in a regular
spread (elliptic linear congruence). See \cite{havl-94a} for a
generalization of this well-known classical result that the regular
spreads of a real $3$-space can be characterized (in the complexified
space) as those sets of lines that join complex-conjugate points of
two skew complex-conjugate lines.
\end{exa}

\begin{exa}
Let $K$ be any field.
Let $U=R=K^n$, with componentwise addition and multiplication.
For $i\in\{1,\ldots,n\}$, let $U_i=Kb_i$, where $b_i$
runs in the standard basis.
Then $U_i$ is a sub-bimodule of $U$, the induced faithful action is the
ordinary action of $K$. Hence the projective model  $\PP(R)^{\Phi}
=\PP(R)$ meets the line
$U_i\times U_i$ in all points. Moreover,
{\em each} $(n-1)$-dimensional
projective subspace
of $\PP(K,U\times U)$ that meets all the lines
$U_i\times U_i$ belongs to $\PP(R)$,  because $\GL_2(R)
\cong \GL_2(K)\times\ldots\times \GL_2(K)$.

 If $n=2$, the set $\PP(R)$ is a
generalization to the not necessarily pappian case
of a {\em hyperbolic linear congruence}.
\end{exa}

\begin{exa}\label{dual}
Let $K$ be any field. Let $U=R=K[\eps]$, where $\eps\not\in K$,
$\eps^2=0$ and $\eps k=k^\alpha \eps$ for some fixed $\alpha\in\Aut(K)$.
This is a ring of {\em twisted dual numbers} over $K$. It is a local
ring with $I=K\eps$ the maximal ideal of all non-invertible elements.
So $U'=I$ is a sub-bimodule of $U=R$, with $\ann(U')=I$, and on
$U'$ we have the induced faithful  representation $(\phi')_f$
of $R/I\cong K$ with $k\eps\cdot a=
ka^\alpha\eps$.
So each point of  $U'\times U'$
is incident with a line of our projective model $\PP(R)=\PP(R)^{\Phi}$.

Now consider the bimodule
$\widetilde U=R/U' \cong K$. The
kernel of the induced representation $\widetilde\phi$
is again $I$. As before, it is easily seen that
each plane through $U'\times U'$
contains a line of $\PP(R)$.

The relation $\notdis$ is an equivalence relation on $\PP(R)$, because
$R$ is a local ring. Easy calculations  show that
elements of $\PP(R)$  belong to the
same equivalence class exactly if they meet $U'\times U'$ in the same
point or, equivalently, if they together with $U'\times U'$  span the
same plane. So there is a bijection $\beta$ between the points of
$U'\times U'$
and the planes through $U'\times U'$  such that for each $p\in \PP(R)$
we have $p\subset (p\cap (U'\times U'))^\beta$. This bijection
$\beta$ is given by
 $K(k^\alpha\eps,l^\alpha\eps)\mapsto K(k,l)\oplus(U'\times U')$.

Moreover, one can compute that the projective model
 $\PP(R)$ consists  of {\em all} lines in $\PP(K,U\times U)$
that meet $U'\times U'$ in a unique point, say $q$, and then lie in the plane
$q^\beta$.

In case $\alpha=\id$ the bijection $\beta$ is a projectivity.
So then the  set $\PP(R)$ is a
generalization of a {\em parabolic linear
congruence}. The ring $R$ is
then the ordinary ring of dual numbers over $K$.
In the general case $\beta$ is only semilinearly induced. If
 $K=\CC$ and $\alpha$ is the complex conjugation, then $R$ is the ring of
{\em Study's quaternions} (see \cite{kar-k+85}, p.445).
\end{exa}

\begin{exa}
Let $R$ be the ring of upper triangular $2\times 2$-matrices with
entries in $K$. Then  $U=K^2$ is in a natural way a faithful
$(K,R)$-bimodule. Moreover, $U'=K(0,1)$ is a sub-bimodule
with $\ann(U')=\{\SMat2{a&b\\0&0}\mid a,b\in K\}$. So $R/\ann(U')\cong K$,
and the induced faithful representation is the ordinary action of $K$
on $U'$. This means that each point of $U'\times U'$ is on a
line of
the projective model $\PP(R)^{\Phi}$.

Now consider $\widetilde U=U/U'$. The kernel of the induced action
is  $J=\{\SMat2{0&b\\0&c}\mid b,c\in K\}$. So  $R/J\cong K$,
and also here we have the ordinary action of $K$ on $\widetilde U\cong K(1,0)$.
Hence each plane through $U'\times U'$
contains a line of $\PP(R)^{\Phi}$.

Up to now, we are in the same situation as in Example \ref{dual}.
An easy calculation shows that the projective model
$\PP(R)^\Phi$ consists of
{\em all} lines that meet $U'\times U'$ in a point.
This is the generalization of a {\em special linear complex} to the not
necessarily pappian case.
\end{exa}

\begin{exa}
Let $U=R=K[\eps,\delta]$
with $\eps\not\in K$, $\delta\not\in K[\eps]$, $\eps,\delta$ central,
and $\eps^2=\delta^2=
\eps\delta=0$. The  projective model
$\PP(R)^{\Phi}=\PP(R)$ is a set of planes in $5$-space.

The ring $R$ is  a local ring with maximal ideal $I=K\eps+K\delta=U'$.
Moreover, $\ann(U')=I$, and $R/I\cong K$ acts on $U'$ componentwise.
So according to \ref{reguli}
the induced model of $\PP(K)$ in the $3$-space $U'\times U'$ is a
regulus $\cR$.

Now consider $\widetilde U=R/U'$. Then $\ker(\widetilde\phi)=I$, and we have the
ordinary faithful action  of $K$ on $\widetilde U\cong K$. So
all hyperplanes ($4$-spaces)
through $U'\times U'$ contain an element of $\PP(R)$.

As in Example~\ref{dual}
the elements of $\PP(R)$ fall into equivalence classes
with respect to $\notdis$, such that equivalent elements have the
same intersection and the same join with $U'\times U'$. This yields a
bijection $\beta$ between  the regulus $\cR$ and the set
of all hyperplanes through $U'\times U'$. As in \ref{dual},
case $\alpha=\id$, this bijection is a
projectivity. A calculation shows that $\PP(R)$ consists of {\em all}
planes that meet the $3$-space $U'\times U'$ in an element of $\cR$, say
$X$, and then lie in the hyperplane~$X^\beta$.
\end{exa}

%{\small

%%%%%%%%%%%%%%%%%%%%%%%%%%%%%%%%%%%%%%%%%%%%%%%%%%%%%%%%%%%%%%%

%\bibliographystyle{plain}
%\bibliography{ketten}

%%%%%%%%%%%%%%%%%%%%%%%%%%%%%%%%%%%%%%%%%%%%%%%%%%%%%%%%%%%%%%%%%%%%
%%%%%%%%%%%%%%%%%%%%%%%%%%%%%%%%%%%%%%%%%%%%%%%%%%%%%%%%%%%%%%%%%%
\bigskip
Institut f\"ur Geometrie\\ Technische Universit\"at\\ Wiedner Hauptstra{\ss}e
8--10\\ A--1040 Wien, Austria
\end{document}